\documentclass[11pt]{amsart}
\usepackage{pstricks,pst-plot}

\parskip=\smallskipamount

\newtheorem{theorem}{Theorem}[section]
\newtheorem{lemma}[theorem]{Lemma}
\newtheorem{corollary}[theorem]{Corollary}

\theoremstyle{definition}
\newtheorem{definition}[theorem]{Definition}
\newtheorem{example}[theorem]{Example}

\newtheorem{remark}[theorem]{Remark}
\newtheorem{question}[theorem]{Question}

\numberwithin{equation}{section}

\newcommand{\B}{\mathbb{B}}
\newcommand{\C}{\mathbb{C}}
\newcommand{\D}{\mathbb{D}}

\newcommand{\N}{\mathbb{N}}

\newcommand{\Z}{\mathbb{Z}}
\renewcommand{\P}{\mathbb{P}}
\newcommand{\R}{\mathbb{R}}

\newcommand\Id{\mathrm{Id}}

\newcommand{\cC}{\mathcal{C}}

\newcommand{\cO}{\mathcal{O}}

\newcommand{\cU}{\mathcal{U}}

\newcommand\wt{\widetilde}
\newcommand\hra{\hookrightarrow}

\newcommand{\Aut}{\mathrm{Aut}\,}

\newcommand\dist{\mathrm{dist}}

\newcommand\bs{\backslash}

\newcommand\wh{\widehat}

\newcommand{\ol}{\overline}

\usepackage{amssymb}
\input xy
\xyoption{all}

\begin{document}
\title[Embeddings of infinitely connected planar domains into $\mathbb C^2$]
{Embeddings of infinitely connected \\ planar domains into $\mathbb C^2$}
\author{Franc Forstneri\v{c}}
\author{Erlend Forn\ae ss Wold}
\address{F.\ Forstneri\v c, Institute of Mathematics, Physics and Mechanics, 
University of Ljubljana, Jadranska 19, 1000 Ljubljana, Slovenia}
\email{franc.forstneric@fmf.uni-lj.si}
\address{E.\ F. Wold, Matematisk Institutt, Universitetet i Oslo,
Postboks 1053 Blindern, 0316 Oslo, Norway}
\email{erlendfw@math.uio.no}
%
%
\date{\today}

\begin{abstract}
We prove that every circled domain in the Riemann sphere admits a 
proper holomorphic embedding into the affine plane $\C^2$.
\end{abstract}

\maketitle

\section{Introduction}
\label{sec:Intro}
It has been a longstanding open problem whether every open 
(noncompact) Riemann surface, in particular, every domain in
the complex plane $\C$, admits a proper holomorphic embedding into 
$\C^2$. (By a {\em domain} we understand a connected open set.)
Equivalently,

{\em Is every open Riemann surface biholomorphic to a smoothly
embedded, topologically closed complex curve in $\C^2$?}

Every open  Riemann surface properly embeds in $\C^3$ and immerses in $\C^2$,
but there is no constructive method of removing 
self-intersections of an immersed curve in $\C^2$. For a history of this subject 
see \cite{ForstnericWold}  and  \cite[\S 8.9--\S 8.10]{F:book}.

In this paper we prove the following general result in this direction.

%
%
%
%
\begin{theorem}\label{main}
Every domain in the Riemann sphere $\P^1=\C\cup\{\infty\}$ with at most countably many 
boundary components, none of which are points, admits a proper holomorphic 
embedding in $\C^2$.
\end{theorem}

By the uniformization theorem of He and Schramm \cite{HeSchramm},
every domain in Theorem \ref{main} is conformally equivalent to a 
{\em circled domain}, that is, a domain whose complement is a 
union of pairwise disjoint closed round discs. 

We prove the same embedding theorem also for generalized circled domains whose 
complementary components are discs and points (punctures), 
provided that all but finitely many of the punctures belong to 
the cluster set of the non-point boundary components 
(see Theorem \ref{main2}). 
In particular, every domain in $\C$ or $\P^1$  
with at most countably many boundary components, 
at most finitely many of which are isolated points, admits a proper
holomorphic embedding into $\C^2$ 
(see Corollary \ref{cor:circled-C} and Example \ref{ex:referee}).

For {\em finitely connected planar domains} without isolated 
boundary points,  Theorem \ref{main} was proved by 
Globevnik and Stens\o nes in 1995 \cite{Globevnik-Stensones}.
More recently it was shown by the authors in \cite{ForstnericWold}
that for every embedded complex curve $\ol C\subset \C^2$, 
with smooth boundary $bC$ consisting of finitely many
Jordan curves, the interior $C=\ol C\bs bC$ admits a proper holomorphic 
embedding in $\C^2$. This result was extended to some 
infinitely connected Riemann surfaces by I.\ Majcen \cite{Majcen-2009}
under  a nontrivial additional assumption on the accumulation set of the boundary curves.
(These results can also be found in \cite[Chap.\ 8]{F:book}.)
Here we do not impose any restrictions whatsoever.

Our proof of Theorems \ref{main} and \ref{main2} is rather involved
both from the analytic as well as the combinatorial point of view, 
something that seems inevitable in this notoriously difficult classical problem.  
Theorem \ref{main} is proved in \S\ref{sec:proof} after 
we develop the technical tools in \S \ref{sec:prel} and \S\ref{sec:keylemma}. 
The main idea is to successively push the boundary components 
of an embedded complex curve in $\C^2$ to infinity by using holomorphic
automorphisms of the ambient space, 
thereby insuring that no self-intersections appear
in the process, while at the same time controlling the convergence 
of the sequence of automorphisms in the interior of the curve. We employ the most 
advanced available analytic tools developed in recent years, 
sharpening further several of them.
A novel part is our inductive scheme of dealing with an 
infinite sequence of boundary components, clustering them 
together into suitable subsets to which the analytic methods
can be applied.  

For simplicity of exposition we limit ourselves to domains in the Riemann sphere, 
although it seems likely that minor modifications yield 
similar results for domains in complex tori. Indeed,
any punctured torus admits a proper holomorphic embedding in $\mathbb C^2$, 
and the uniformization theory of He and Schramm \cite{HeSchramm} 
applies in this case as well.
For infinitely connected domains in Riemann surfaces of genus $>1$
the main problem is to find a suitable initial embedding of the 
uniformized surface into $\C^2$. One of the difficulties in working
with non-uniformized boundary components is indicated in 
Remark \ref{rem:closed}; another one can be seen in the
last part of proof of Lemma \ref{remove} which is a key 
ingredient in our construction.

Casting a view to the future, what is now needed to approach 
the general embedding problem is some progress 
on embedding \emph{punctured} Riemann surfaces into $\C^2$.
It is plausible that a method for answering 
the following question in the affirmative would 
lead to a complete solution to the embedding 
problem for planar domains with countably many
boundary components.

\begin{question}
Assume that $f\colon \ol{\mathbb D}\rightarrow\mathbb C^2$ is 
a holomorphic embedding, $K\subset\mathbb C^2\setminus f(b\mathbb D)$
is a compact polynomially convex set, $C\subset\mathbb D$ is a 
compact set with $f^{-1}(K)\subset  \mathring C$, 
and $a\in\mathbb D\setminus C$ is a point.  
Is $f$ uniformly approximable on $C$ by proper 
holomorphic embeddings $g\colon \ol{\mathbb D}\setminus\{a\}\hra \mathbb C^2$
satisfying 
\[
	 g^{-1}(g(\ol{\mathbb D}\setminus\{a\})\cap K)\subset \mathring C \ ?
\]
\end{question}

In another direction, one can ask to what extent does 
Theorem \ref{main} hold for domains in $\P^1$ 
with uncountably many boundary components.
A quintessential example of this type is a Cantor set,
i.e., a compact, totally disconnected, perfect set.
Recently Orevkov \cite{Orevkov} constructed an example of 
a Cantor set $K$ in $\C$ whose complement $\C\setminus K$
embeds properly holomorphically in $\C^2$.  (See also \cite[Theorem 8.10.4]{F:book}). 
His method, using compositions of rational shears of $\C^2$, 
does not seem to apply to a specific Cantor set. 
The methods explained in this paper 
offer some hope for future developments as indicated by Theorem \ref{main2} and 
Example \ref{ex:referee} below. 

The problem of embedding an open Riemann
surface in $\C^2$ is purely complex analytic, and there are no topological obstructions.
Indeed, Alarcon and Lopez \cite{Alarcon-Lopez}
recently proved that every open Riemann surface $X$ contains a 
domain $\Omega\subset X$, homeotopic to $X$, which embeds properly 
holomorphically in $\mathbb{C}^2.$ In particular, every
open orientable surface admits a smooth proper embedding 
in $\C^2$ whose image is a complex curve.

%
\section{Preliminaries}
\label{sec:prel}
In this and the following section we prepare the technical tools
that will be used in the proof. 
The main result of this section, Theorem \ref{exposing}, 
gives holomorphic embeddings of bordered 
Riemann surfaces into $\C^2$ with exposed wedges 
at finitely many boundary points.

We begin by introducing the notation. Let $\P^1=\C\cup\{\infty\}$
be the Riemann sphere. We denote by $\D=\{z\in \C\colon |z|<1\}$ the open unit disc
and by $\D_r =\{|z|<r\}$ the disc of radius $r$ centered at the origin. 
Let $(z_1,z_2)$ be complex coordinates on $\C^2$, and let
$\pi_i\colon \C^2 \to \C$ denote the coordinate
projection $\pi_i(z_1,z_2)=z_i$ for $i=1,2$.
We denote by $\mathbb B_r$ and ${\ol \B}_r$ the open and
the closed ball in $\C^2$, respectively, of radius $r$ and centered at the origin. 
Let $\Aut\C^2$ denote the group of all holomorphic automorphisms of $\C^2$.
By $\Id$ we denote the identity map; its domain will always be 
clear from the context.
We denote by $\wh L$ the polynomial hull of a compact set $L\subset \C^n$.

%
%
%
%
\begin{definition}
\label{def:circled}
A domain $\Omega\subset\mathbb P^1$ is said to be a
\emph{circled domain} if the complement 
$\mathbb P^1\setminus\Omega\ne \emptyset$ is a  
union of pairwise disjoint closed round discs 
$\ol\triangle_j\subset \mathbb P^1$ of positive radii. 
\end{definition}

Clearly a circled domain has at most countably many 
complementary discs. Mapping one of them onto $\P^1\setminus \D$
by an automorphism of $\P^1$ (a fractional linear map) 
we see that a circled domain can be thought of as being contained 
in the unit disc $\D$. 

The next lemma, and the remark following it, 
will serve to cluster together certain complementary 
components into finitely many discs;  this will enable the
use of holomorphic automorphisms for pushing these components
towards infinity in the inductive process.

%
%
%
%
\begin{lemma}\label{division}
Let $\Omega\subset\mathbb P^1$ be a domain, 
let $K\subset\mathbb P^1\setminus\Omega$
be a closed set which is a union of 
complementary connected components of $\Omega$, 
and let $U\subset\mathbb P^1$ be 
an open set containing $K$.  Then 
there exist finitely many pairwise disjoint, smoothly bounded discs
$\overline{D}_j\subset U$ $(j=1,\ldots,m)$ such that 
\[
	K\subset \cup_{j=1}^m D_j, \qquad 
	bD_j\cap (\mathbb P^1\setminus\Omega) =\emptyset\ \; \text{for}\ j=1,\ldots,m.
\]
\end{lemma}

\begin{proof}
Let $K_j\subset \mathring K_{j+1}\subset K_{j+1}$ be 
an exhaustion of $\Omega$ by smoothly bounded connected compact 
sets $K_j$.   Then $\mathbb P^1\setminus K_j$
is the union of finitely many discs $\mathcal U_j=\{U^j_1,\ldots,U^j_{m(j)}\}$
for each $j$. Clearly $\mathcal U_j$ is a cover of $K$, 
and we claim that if $j$ is large enough then 
$\mathcal U_j$ contains a subcover whose union is relatively 
compact in $U$.  Otherwise there would exist  
a sequence $U^j_{k(j)}\supset U^{j+1}_{k(j+1)}$
of discs such that $U^j_{k(j)}\cap K\neq\emptyset$
and $\overline{U^j_{k(j)}}\cap\mathbb (\mathbb P^1\setminus U)\neq\emptyset$
for each $j$; but then $\cap_{j=1}^\infty\overline{U^{j}_{k(j)}}$
would be a connected complementary component of $\Omega$ which 
is contained in $K$ and which intersects $\mathbb P^1\setminus U$,
a contradiction.  Hence for $j$ large enough the discs $D_1,\ldots, D_m$ 
in $\cU_j$ satisfy the stated properties.
\end{proof}

%
%
%
%
\begin{remark}
\label{rem:closed}
When applying Lemma \ref{division}  to prove Theorem \ref{main},
it will be crucial that if $\Omega\subset\mathbb P^1$
is a circled domain with complementary discs $\overline\triangle_j$,
and if $C\subset\mathbb P^1$ is any compact set, then 
the union of all $\overline\triangle_j$'s intersecting
$C$ is a closed set which is a union of complementary 
connected components of $\Omega$. The proof is elementary and is left to the reader.
However, this fails in general if discs are replaced
by more general connected closed sets. This is one of the reasons
why our proof of Theorem \ref{main} does not apply (at least
not directly) to domains in compact Riemann surfaces of genus $>1$.
\qed \end{remark}

%
%
%
%
\begin{definition}
\label{def:wedge}
Let $0<\theta<2\pi$. A domain $\Omega\subset\mathbb C$
is an (open) {\em $\theta$-wedge with vertex} $a\in b\Omega$ if 
there exist a $\cC^2$ map of the form 
\[
	\varphi(\zeta)=a + \lambda \zeta + O(|\zeta|^2),\quad \lambda\ne 0
\]
in a neighborhood of the origin $0\in\C$, and for every sufficiently
small $\epsilon>0$ a neighborhood $U_\epsilon\subset \C$ of 
the point $a$ such that 
\[
		U_\epsilon \cap\Omega=\varphi \bigl(\{\zeta \in\mathbb C^* \colon 
		0< \arg(\zeta)<\theta,\ 0<|\zeta|<\epsilon\} \bigr).
\]
The closure of an open wedge will be called a {\em closed wedge}.
\end{definition}

If $\Omega$ is a domain in a Riemann surface $Y$, we apply the same 
definition of a $\theta$-wedge 
in a local holomorphic coordinate near the point $a\in b\Omega \subset Y$.
In particular, if $\Omega\subset \P^1=\C\cup\{\infty\}$ 
and $a=\infty \in b\Omega$, we apply the definition 
in the local chart $z\to 1/z$ on $\P^1$ mapping $\infty$ to $0$.

%
%
%
%
Given a nonempty subset $E$ of $\C^2$ and a linear projection $\pi\colon \C^2\to\C$,
a point $p\in E$ is said to be {\em $\pi$-exposed}, 
and $E$ is said to be {\em $\pi$-exposed at the point $p$}, if 
\begin{equation}
\label{eq:exposed}
	E \cap \pi^{-1}(\pi(p)) = \{p\}. 
\end{equation}

Recall that a {\em bordered Riemann surface} is a compact one 
dimensional complex manifold, $\overline X$, 
with boundary $bX$ consisting of finitely many
Jordan curves. The interior $X$ of a bordered Riemann surface 
is biholomorphic to a relatively compact, smoothly bounded domain 
in a Riemann surface $Y$.

We shall use  the following notion of an {\em exposed $\theta$-wedge}.

%
%
%
%
\begin{definition}
\label{def:exposed-wedge}
Let $X$ be a bordered Riemann surface, embedded as a 
smoothly bounded relatively compact 
domain in a Riemann surface $Y$. 
Pick a point $a\in bX$ and a number $\theta\in (0,2\pi)$.
An injective continuous map $f\colon \overline X\hra \mathbb C^2$ 
is said to be a 
{\em holomorphic embedding with a $\pi_1$-exposed $\theta$-wedge at $f(a)$} 
if $f$ is holomorphic in $X$, and there exists an open neighborhood 
$U$ of $a$ in $Y$ such that 
\begin{itemize}
\item[\rm (i)] 
	the domain $\Omega= (\pi_1 \circ f)(U\cap X) \subset \C$ is a 
	$\theta$-wedge with vertex $\pi_1(f(a))$ (see Def.\ \ref{def:wedge}), 
\item[\rm (ii)]
	$f(\ol{U\cap X})$ is a smooth graph over $\ol\Omega$ 
	that is holomorphic over $\Omega$, and 
\item[\rm (iii)]  
$\pi_1^{-1}({\ol\Omega})\cap f(\overline X)=f(\ol{U\cap X})$.
\end{itemize}
If the domain $\Omega\subset \C$ is instead smooth near 
the point $\pi_1(f(a))\in b\Omega$, we say that 
$f$ is a holomorphic embedding which is $\pi_1$-exposed at $f(a)$.
\end{definition}

%
%
%
%
\begin{remark} {\bf (On terminology.)} 
\label{corners}
We shall consider embeddings $f\colon \ol X\hra \C^2$ that are 
holomorphic in the interior $X$ and smooth on $\ol X$, except at finitely
many boundary points where $f(X)$ has (exposed) wedges in the sense of the above definition. 
Any such map will be called  a {\em holomorphic embedding with corners}.
We shall use embeddings with corners of a particular type: 
If $X$ is a smoothly bounded, relatively compact domain in a Riemann surface $Y$,
we will construct holomorphic embeddings $\wt f \colon Y \hra \mathbb C^2$ 
and injective continuous maps $\varphi \colon \ol X\to Y$, holomorphic on $X$
and smooth at all but finitely many boundary points $a_j\in bX$, such that 
\begin{equation}\label{type}
	f :=\wt  f\circ\varphi\colon \ol X\hookrightarrow\C^2 \ 
	\mbox{ is an embedding with corners at the }a_j\mbox{'s}.
\end{equation}
In the sequel we will refer to such maps simply as {\em being of the form} (\ref{type}).
The precise choice of the Riemann surface $Y$ will not 
be important, and we will allow $Y$ to shrink around $\ol X$
without saying it every time. 
\qed \end{remark}

The following lemma shows how to create wedges at 
smooth boundary points of a domain in a Riemann surface. 

%
%
%
%
\begin{lemma}
\label{wedging}
 Let $X\Subset Y$ be Riemann surfaces, and assume that 
 $bX$ is smooth outside a finite 
 set of points.  Let $a_1,\ldots,a_m\in bX$, $b_1,\ldots,b_k\in\overline X$
 be distinct points, with $bX$ smooth near the $a_j$'s,
 and let $\theta_j\in (0,2\pi)$ for $j=1,\ldots,m$.  Then 
 there exists a sequence of injective continuous maps 
 $\varphi_i \colon \overline X\rightarrow Y$, 
 holomorphic on $X$ and smooth on $\overline X\setminus\{a_1,\ldots,a_m\}$,
 satisfying the following properties:
 \begin{itemize}
 \item[(1)] $\varphi_i\rightarrow\Id$ uniformly on $\overline X$ as $i\rightarrow\infty$,
 \item[(2)] $\varphi_i(a_j)=a_j$ and $\varphi_i(X)$ is a $\theta_j$-wedge with vertex 
   $a_j$ $(j=1,\ldots,m)$, and
\item[(3)] $\varphi_i(x)=b_j + o(\dist(x,b_j)^2)$ as $x\rightarrow b_j$ $(j=1,\ldots,k)$.
 \end{itemize}
\end{lemma}

\begin{proof}
The proof is similar to that of Lemma 8.8.3 in \cite[p.\ 366]{F:book},
and it will help the reader to consult Fig.\ 8.1 in \cite[p.\ 367]{F:book}.

 By enlarging the domain $X$ slightly away from the $a_j$'s we may assume 
 that $X$ is smoothly bounded. For simplicity of notation we explain 
 the proof in the case when there is only one such point $a=a_1$;
 the same procedure can be performed simultaneously 
 at finitely many points. 
 
 Choose a smoothly bounded disc $D$ in $Y$ such that $a\in bD$,
 $\ol D$ does not contain any of the points $b_j$, and  
 $\ol{U} \cap \ol X \setminus\{a\} \subset D$ holds for some 
 small open neighborhood $U$ of the point $a$ in $Y$. 
 (The disc $D$ is obtained by pushing the boundary of 
 $X$ slightly out near $a$ and then rounding off.) 
 We also choose a compact Cartan pair $(A,B)\subset Y$
 with $\overline X\subset (A\cup B)^\circ$ and 
 $C:=A\cap B\subset D$. (For the notion of a Cartan pair 
 see \cite[Def.\ 5.7.1]{F:book}.) The set $A$ is
 chosen such that it contains a neighborhood of $a$, 
 and $B$ contains $\ol X\setminus U'$ for a small neighborhood $U'\subset U$ 
 of the point $a$.  
 
 The Riemann Mapping Theorem furnishes a sequence of injective 
 continuous maps  $\psi_i \colon \overline D\rightarrow Y$ that are
 holomorphic in $D$ and smooth on $\ol D\setminus \{a\}$ such that 
 $\psi_i(a)=a$, $\psi_i(D)$ is a $\theta_1$-wedge with vertex $a$
 (see Def.\ \ref{def:wedge}), and the sets $\psi_i(\ol D)$ converge to 
 $\ol D$ as $i\to\infty$.  We may assume that $\psi_i\rightarrow\Id$
 uniformly on $\overline D$ (see Goluzin \cite[Theorem 2, p.\ 59]{Goluzin}).  
 This implies that $\psi_i(C)\subset D$ for all sufficiently large $i\in \N$.
 
 By Theorem 8.7.2 in \cite[p.\ 359]{F:book}  there exist an integer $i_0\in \N$ 
 and sequences of injective holomorphic maps $f_i \colon A\rightarrow Y$ and 
 $g_i\colon B\rightarrow Y$ $(i\ge i_0)$,  both converging to the identity map 
 and tangent to the identity to second order  at those points 
 $a$ and $b_j$ which are contained in their respective domains, such that 
 \[
 	\psi_i \circ f_i = g_i \ \ \text{holds on} \ C.  
 \]
 The sequence of maps $\varphi_i\colon \ol X\to Y$, defined by 
 \[
 	\varphi_i=\psi_i\circ f_i \ \text{on} \ A\cap\overline X,
 	\qquad \varphi_i=g_i \ \text{on}\ \overline X\cap B
 \] 
 then satisfies the conclusion of the lemma. Injectivity of $\varphi_i$ on $\ol X$
 for sufficiently large index $i$ can be seen exactly as in the proof
 of \cite[Lemma 8.8.3]{F:book} (see bottom of page 359 in the
 cited source).
\end{proof}

Using Lemma \ref{wedging} we obtain the 
following version of the main tool introduced in 
\cite{ForstnericWold} for exposing boundary points of bordered Riemann surfaces. 
(See also Theorem 8.9.10 and Fig.\ 8.2 in \cite[pp.\ 372--373]{F:book}.)
The main novelty here is that we create {\em exposed points with wedges}.

%
%
%
%
\begin{theorem}
\label{exposing}
Let $\ol X$ be a smoothly bounded domain in a Riemann surface $Y$,  
$f\colon \overline X\hookrightarrow\mathbb C^2$ a holomorphic embedding
with corners of the form (\ref{type}), and $a_1,\ldots,a_m \in bX$, 
$b_1,\ldots,b_k\in\overline X$ distinct points 
such that $f$ is smooth near the $a_j$'s.  
Let $\gamma_j \colon [0,1]\rightarrow\mathbb C^2$ $(j=1,\ldots,m)$
be smooth embedded arcs with pairwise disjoint images satisfying 
the following properties:
\begin{itemize}
\item $\gamma_j([0,1]) \cap f(\overline X)=\gamma_j(0)=f(a_j)$ for $j=1,\ldots,m$, and
\item the image $E:=f(\overline X) \,\cup\, \cup_{j=1}^m \gamma_j([0,1])$ 
is $\pi_1$-exposed at the point $\gamma_j(1)$ for $j=1,\ldots,m$
(see (\ref{eq:exposed})).
\end{itemize}
Given an open set $V\subset\mathbb C^2$ containing 
$\cup_{j=1}^m \gamma_j([0,1])$, 
an open set $U\subset Y$ containing the points $a_j$ 
and satisfying $f(\ol {U\cap X})\subset V$, and numbers $0<\theta_j<2\pi$
$(j=1,\ldots,m)$ and $\epsilon>0$, there exists a holomorphic embedding 
with corners $g\colon \overline X\hookrightarrow\mathbb C^2$ 
of the form (\ref{type}) satisfying the following properties: 
\begin{itemize}
\item[\rm (1)] $\|g-f\|_{\overline X\setminus U}<\epsilon$, 
\item[\rm (2)] $g( \ol{U\cap X})\subset V$, 
\item[\rm (3)] $g(x)=f(x) + o(\dist(x,b_j)^2)$ as $x\to b_j$ $(j=1,\ldots,k)$, 
\item[\rm (4)] $g(a_j)=\gamma_j(1)$ and $g(\overline X)$ is 
$\pi_1$-exposed with a $\theta_j$-wedge at $g(a_j)$ for 
every $j=1,\ldots,m$, and
\item[\rm (5)] $g$ is smooth near all points $x\in bX\setminus \{a_1,\ldots,a_m\}$
at which $f$ is smooth.
\end{itemize}
\end{theorem}

If for some $j\in \{1,\ldots,k\}$ we have that 
$b_j\in bX$ and $f(X)$ is a wedge at the point $f(b_j)$,
then property (3) insures that $g(X)$ remains a wedge with the same angle
at $f(b_j)=g(b_j)$. In addition, property (4) insures 
that $g(X)$ is an exposed wedge at each of the points $g(a_j)$.

\begin{proof}
Since $f$ is of the form (\ref{type}), we write $f=\wt f\circ\varphi$
where $\wt f\colon Y\hra\C^2$ is a holomorphic embedding. Set $X'=\varphi(X)\Subset Y$.
Lemma \ref{wedging}, applied to the domain $X'$ and the 
points $a_j'= \varphi(a_j) \in bX'$, $b'_j=\varphi(b_j)\in \ol{X'}$,
gives an injective continuous map $ \psi\colon \overline{X'} \rightarrow Y$
close to the identity map, with $ \psi$ holomorphic on $X'$ 
and smooth on $\overline X'\setminus\{a'_1,\ldots,a'_m\}$,
such that 
\begin{itemize}
\item[\rm (2')] $ \psi(a'_j)=a'_j$ and $ \psi(X')$ is a $\theta_j$-wedge 
with vertex  $a'_j$ ($j=1,\ldots,m$), and
\item[\rm (3')] $ \psi(x)=b'_j + o(\dist(x,b'_j)^2)$ as $x\rightarrow b_j$ $(j=1,\ldots,k)$.
\end{itemize}
(The map $\psi$ is one of the maps $\varphi_i$ in Lemma \ref{wedging}, 
and the properties (2'), (3') correspond to (2), (3) in that lemma, respectively.)

Set $\wt \varphi=\psi\circ\varphi \colon \ol X\to Y$; this is an embedding with 
the analogous properties as $\varphi$, but with additional $\theta_j$-wedges at 
the points $a'_j\in bX'$. 
The embedding with corners $\wt f\circ \wt \varphi\colon \ol X\hra  \C^2$ 
then satisfies properties (1)--(3) and (5) (for the map $g$) in Theorem \ref{exposing}.

In order to achieve also condition (4) we apply Theorem 8.9.10 in \cite{F:book}
and the proof thereof. 
(The original source for this result is \cite[Theorem 4.2]{ForstnericWold}.)
We recall the main idea and refer to the cited works for the details. 
By pushing the boundary $bX'$ slightly outward away from the $a'_j$'s 
we obtain a smoothly bounded domain $Z \Subset Y$ such that 
$\ol X' \subset Z \cup\{a'_1,\ldots,a'_m\}$.
We attach to $\ol Z$ short pairwise disjoint 
embedded arcs $\Gamma_j \subset Y$ intersecting $\overline Z$ only 
at the points $a'_j$.  By Mergelyan's theorem we can change the embedding 
$\wt f$ so that it maps the arc $\Gamma_j$ approximately onto the 
arc $\gamma_j$ for each $j=1,\ldots,m$, taking the other endpoint $c_j$
of $\Gamma_j$ to the exposed endpoint $\gamma_j(1)\in\C^2$ 
of $\gamma_j$ and remaining close to the initial embedding on $\ol Z$.  
At each point $a'_j \in bZ$ we choose a small smoothly bounded disc 
$D_j\subset Y$ with the same properties as in the proof of Lemma \ref{wedging};
in particular, $a'_j \in bD_j$ and $D_j$ contains $\ol Z\setminus \{a'_j\}$
near the point $a'_j$. By the Riemann Mapping Theorem 
we find for each $j\in\{1,\ldots,m\}$ a holomorphic map 
$h_j\colon \overline D_j\rightarrow Y$ stretching $\overline D_j$ to contain 
the arc $\Gamma_j$, mapping $a'_j$ to the other endpoint 
$c_j$ of $\Gamma_j$ and remaining close to the 
identity except very near the point $a'_j$.  
We then glue the $h_j$'s to an approximation of the identity 
map on the rest of the domain $\overline Z$, using again 
Theorem 8.7.2 in \cite[p.\ 359]{F:book}.
This gives an injective holomorphic map $h\colon \wt Y \hra Y$ 
in an open neighborhood $\wt Y$ of $\ol Z$ such that $h|_{\ol Z}$ 
is close to the identity, except very near the points $a'_j\in bZ$. 
The holomorphic embedding $\wt g:=\wt f\circ h\colon \wt Y \hra \C^2$ 
is then close to $\wt f$ on $\ol Z$, except near the points $a'_j$. 
By the construction, $\wt g(a'_j)$ is a $\pi_1$-exposed point 
of $\wt g(\ol Z)$ for $j=1,\ldots,m$. 
The embedding with corners $g=\wt g \circ \wt \varphi\colon \ol X\hra  \C^2$ 
is then of the form (\ref{type}) and satisfies 
properties (1)--(5) in Theorem \ref{exposing}.
\end{proof}

%
\section{The main lemma}
\label{sec:keylemma}
In this section we prove the following key lemma 
that will be used in the proof of Theorem \ref{main}. 
It is similar in spirit to Lemma 1 in \cite[p.\ 4]{Wold}
(see also \cite[Lemma 4.14.4., p.\ 150]{F:book}),
but with improvements needed to deal with the more 
complicated situation at hand.

%
%
%
%
\begin{lemma}
\label{remove}
Let $\Omega= \P^1 \setminus\cup_{j=0}^\infty \ol\triangle_j$ be a circled domain,
and let $\Omega'= \P^1 \setminus\cup_{j=0}^k \ol\triangle_j$ for some $k\in\N$.
Pick a point $c_j\in b\triangle_j$ for $j=0,1,\ldots,k$.
Assume that $f\colon \overline{\Omega'} \hra \mathbb C^2$ is
a holomorphic embedding with a $\pi_1$-exposed $\theta_j$-wedge 
at each point $f(c_j)$ and $\theta_0+\cdot\cdot\cdot+\theta_k<2\pi$.
Let $g$ be a rational shear map of the form 
\[
	g(z_1,z_2)=\left( z_1,z_2 + \sum_{j=0}^k\frac{\beta_j}{z_1-\pi_1(f(c_j))} \right).
\]
Assume that there exist open neighborhoods $U_j\subset \P^1$ of 
the points $c_j$ such that $(\pi_2\circ g\circ f)(U_j)\subset \P^1$ are 
$\theta_j$-wedges whose closures only intersect at their common vertex 
$\infty\in \P^1$. (This can be arranged by a suitable choice of the 
arguments of the numbers $\beta_j$, while at the same time 
keeping $|\beta_j|>0$ arbitrary small.)
Given a compact polynomially convex set $K\subset\mathbb C^2$ with 
\[
	K\cap (g\circ f) (b\Omega' \cup 
	(\cup_{i=k+1}^\infty\overline\triangle_i)) =\emptyset
\]
and numbers $N\in\mathbb N$ and $\epsilon>0$, 
there exists a $\psi\in\Aut\mathbb C^2$ such that 
\begin{itemize}
\item[(1)] 
$(\psi\circ g\circ f)(b\Omega'\cup (\cup_{i=k+1}^\infty \ol\triangle_i))
\subset \mathbb C^2\setminus \ol{\mathbb B}_N$, and 
\item[(2)] $\|\psi - \Id\|_K<\epsilon$. 
\end{itemize}
\end{lemma}

\begin{proof}
We may assume that $\triangle_0= \P^1\setminus \ol\D$, so 
$\Omega=\D\setminus \cup_{j=1}^\infty \ol\triangle_j$.
By increasing the number $N\in\N$ we may also assume that $K\subset\mathbb B_N$.

Set $X=(g\circ f)(\Omega')$, 
$\gamma_j=(g\circ f)(b\triangle_j\setminus \{c_j\})$ 
$(j=0,\ldots,k)$, and $\gamma=\cup_{j=0}^k \gamma_j$. 
Then $\ol X$ is an embedded bordered Riemann surface in $\C^2$
whose boundary $bX=\gamma$ consists of pairwise 
disjoint properly embedded real curves $\gamma_j$ diffeomorphic to $\R$,
and the second coordinate projection $\pi_2\colon\ol X\to\C$ is proper.
Let $\triangle'_i=(g\circ f)(\triangle_i) \subset X$ for $i=k+1,k+2,\ldots$;
then 
\[
	X\setminus \cup_{i=k+1}^\infty \ol \triangle'_i = (g\circ f)(\Omega).
\]

To prove the lemma we must find an automorphism 
$\psi\in\Aut\C^2$ sending the boundary curves $bX=\gamma$ and all the
discs $\ol\triangle'_i$ for $i>k$ out of the ball $\ol \B_N$, while at the
same time approximating the identity map on the compact set $K$. 
We seek $\psi$ of the form 
\[
	\psi=\phi_1\circ \phi_2, \quad \text{where}\ \phi_1,\phi_2\in\Aut\C^2.
\]
We begin by constructing $\phi_1$.

The conditions on $f$ and $g$ insure that for any sufficiently large disc 
$D\subset \C$ centered at the origin the projection 
$\pi_2\colon \ol X\setminus \pi_2^{-1}(D) \to \C\setminus D$ is 
injective and maps $\ol X\setminus \pi_2^{-1}(D)$ onto the union 
of $k+1$ pairwise disjoint wedges with the common vertex at $\infty$;
furthermore, the closed set
\begin{equation}
\label{eq:Dandcurves}
	\overline D \cup \pi_2\left(
	\gamma \,\cup\, \cup_{i=k+1}^\infty \overline\triangle'_i \right) \subset\C
\end{equation}
can be exhausted by polynomially convex compact sets. 
To see this, note that if $V_j'\subset V_j$ are small round discs in 
$\C$ centered at the point $c_j$ such that $V_j\subset U_j$
for $j=0,1,\ldots,k$, where the $U_j$'s satisfy the hypotheses 
of the lemma, then the sets 
\[
		(bV_j\setminus \triangle_j)
 		\cup \bigl( b\triangle_j\cap (\ol V_j\setminus V_j') \bigr) 
		\cup \left( \cup_{i=k+1}^\infty \ol\triangle_i\cap (\ol V_j\setminus V_j') \right) 
		\subset \C
\]
are polynomially convex, and the map 
$
	\pi_2\circ g\circ f \colon \cup_{j=0}^k \ol V_j\cap\Omega'\rightarrow\mathbb C
$
is an injection onto a union of wedges such that the closures of any two 
of them intersect only at their common vertex at $\infty$. 
An exhaustion of the set (\ref{eq:Dandcurves}) by polynomially 
convex compact sets is constructed by letting the radii of 
the discs $V'_j$ going to $0$.

Let $J=\{i\in\N\colon i \geq k+1,\ \pi_2(\overline\triangle'_i)\cap\overline D\neq\emptyset\}$. 
Consider the compact set
\[
 	C := \left[\gamma\cap\pi_2^{-1}(\overline D)\right]
 	\cup \left[\cup_{i\in J} \overline\triangle'_i \right] \subset \ol X.
\]
(Fig.\ \ref{fig:C} below shows $C$ with bold lines
and black discs.) We claim that $C$ is polynomially convex.  
Clearly $C$ is holomorphically convex in $\ol X$ since its complement
is connected. Furthermore, $\ol X$ can be exhausted by 
compact smoothly bounded subdomains $X_j\subset \ol X$ such that each 
boundary component of $X_j$ intersects the boundary of $X$.
(It suffices to take the intersection of $\ol X$ with a sufficiently 
large ball and smoothen the corners.) 
Then $\widehat X_j\setminus X_j$ is either
empty or a pure one-dimensional complex subvariety of $\mathbb C^2\setminus X_j$
(see Stolzenberg \cite{Stolz}), the latter 
being impossible since the variety would have to be unbounded.  
Hence every such set $X_j$ is polynomially convex, and by choosing 
it large enough to contain $C$ we see that $C$ is polynomially convex.

We will construct $\phi_1$ as a composition $\phi_1=\sigma_2\circ\sigma_1 \in\Aut \C^2$ 
which is close to the identity on $K$ and satisfies 
$\phi_1(C)\subset\C^2\setminus\overline \B_N$;
equivalently, $C\cap \phi_1^{-1}(\ol \B_N)=\emptyset$.

%
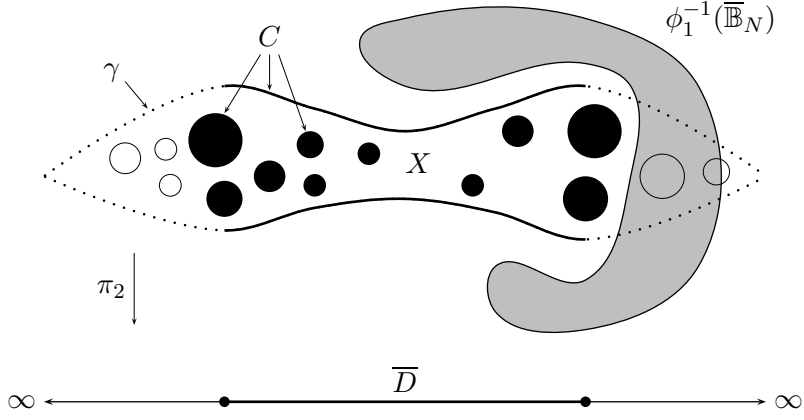
\begin{figure}[ht]
\psset{unit=0.6cm, linewidth=0.5pt}  
\begin{pspicture}(-9,-1)(9,8.5)

%
%
\pscustom[fillstyle=solid,fillcolor=lightgray]
{ \psecurve(0,6)(-1,7)(0,7.5)(4,7.5)(6.5,6)(6.5,2)(4,0.6)(2,1)
(2,2)(4,1.7)(5,4)(5,6)(4,6.5)(0,6)(-1,7)(0,7.5)
}

\rput(7,7.5){$\phi_1^{-1}(\ol \B_N)$}

%
%
\psline[linewidth=0.5pt, arrows=<->](-8,-1)(8,-1)
\rput(-8.5,-1){$\infty$}
\rput(8.5,-1){$\infty$}
\psline[linewidth=1pt, arrows=*-*](-4,-1)(4,-1)
\rput(0,-0.5){$\ol D$}

\psline[linewidth=0.2pt, arrows=->](-6,2.3)(-6,0.7)
\rput(-6.5,1.5){$\pi_2$}

%
%
\pscurve[linestyle=dotted,linewidth=1pt](-8,4)(-4,6)(-2,5.5)(0,5)(2,5.5)(4,6)(8,4)
\pscurve[linestyle=dotted,linewidth=1pt](-8,4)(-4,2.8)(-2,3.3)(0,3.5)(2,3.2)(4,2.6)(8,4)
%
%
\psecurve[linewidth=1pt](-8,4)(-4,6)(-2,5.5)(0,5)(2,5.5)(4,6)(8,4)
\psecurve[linewidth=1pt](-8,4)(-4,2.8)(-2,3.3)(0,3.5)(2,3.2)(4,2.6)(8,4)
\rput(-6.5,6.3){$\gamma$}
\psline[linewidth=0.2pt,arrows=->](-6.2,6.1)(-5.7,5.5)
\rput(0.3,4.3){$X$}

%
%
\pscircle*(-4.2,4.8){0.6}
\pscircle*(4.2,5){0.6}
\pscircle*(-0.8,4.5){0.25}
\pscircle*(-2,3.8){0.25}
\pscircle*(-2.1,4.7){0.3}
\pscircle*(2.5,5){0.35}
\pscircle*(1.5,3.8){0.25}
\pscircle*(4,3.5){0.5}
\pscircle*(-3,4){0.35}
\pscircle*(-4,3.5){0.4}

%
%
\pscircle[linewidth=0.3pt](5.7,4){0.5}
\pscircle[linewidth=0.3pt](6.9,4.1){0.3}
\pscircle[linewidth=0.3pt](-6.2,4.4){0.35}
\pscircle[linewidth=0.3pt](-5.2,3.8){0.25}
\pscircle[linewidth=0.3pt](-5.3,4.6){0.25}

\rput(-3,7.1){$C$}
\psline[linewidth=0.2pt,arrows=->](-3.2,6.7)(-4,5.4)
\psline[linewidth=0.2pt,arrows=->](-3,6.7)(-3,5.9)
\psline[linewidth=0.2pt,arrows=->](-2.8,6.7)(-2.2,5)

\end{pspicture}
\caption{The set $C$.}
\label{fig:C}
\end{figure}

By \cite[Lemma 1]{Wold} (see also \cite[Corollary 4.14.5]{F:book})
there exists $\sigma_1\in\Aut\mathbb C^2$ which is close to the identity on $K$ 
and satisfies $\sigma_1(\gamma) \subset \mathbb C^2\setminus \overline{\mathbb B}_N$.

Let $K'$ be the union of all discs $\overline\triangle_i$ $(i\in J)$ 
whose images $\overline\triangle'_i$ 
satisfy 
\[
	 \sigma_1(\overline\triangle'_i)\cap\overline{\mathbb B}_N\neq\emptyset.
\]
Since $\sigma_1(\gamma)\cap {\ol \B}_N=\emptyset$,
the set $(\sigma_1\circ g\circ f)^{-1}({\ol\B}_N) \subset \Omega'$ is  
compact, and hence $K'$ is also compact (see Remark \ref{rem:closed}). 
Lemma \ref{division} gives pairwise disjoint smoothly bounded discs 
$D_1,\ldots,D_m$ in $\Omega'$ whose union 
$\cup_{j=1}^m D_j$ contains $K'$ and whose closures 
$\ol D_j$ avoid  $b\Omega' \cup (g\circ f)^{-1}(K)$.  
Set $D'_j=(g\circ f)(D_j) \subset X$ for $j=1,\ldots,m$. 
The set 
\[
		L := K\cup (C \setminus \cup_{j=1}^m \ol D'_j) \subset \C^2 
\] 
is then polynomially convex (argue as above for the set $C$, using the 
fact that $K$ is disjoint from $C$). 
The union of discs $E_0:=\cup_{j=1}^m \sigma_1(\overline D'_j)$ 
is polynomially convex and disjoint from $ \sigma_1(L)$, so it
can be moved out of the ball ${\ol\B}_N$ by a holomorphic isotopy 
in the complement of the polynomially convex set $\sigma_1(L)$.
(It suffices to first contract each disc $\sigma_1(\overline D'_j)$
into a small ball around one of its points and then move 
these small balls out of the set $\sigma_1(L)$ 
along pairwise disjoint arcs.) Furthermore, letting $E_t \subset \C^2$ $(t\in [0,1])$ 
denote the trace of $E_0$ under this isotopy, we can insure that 
for every $t$ the union $E_t\cup \sigma_1(L)$ is polynomially convex.
The Anders\'en-Lempert theory 
(cf.\ \cite[Theorem 4.12.1]{F:book}) now furnishes   
an automorphism  $\sigma_2\in\Aut\mathbb C^2$
which is close to the identity on the set $\sigma_1(L)$ and satisfies 
\[
	(\sigma_2\circ\sigma_1)(\cup_{j=1}^m \overline D'_j)\subset\C^2\setminus\ol{\B}_N.
\]
The automorphism $\phi_1=\sigma_2\circ\sigma_1 \in\Aut\C^2$ is then close to the 
identity map on $K$ and satisfies $\phi_1(C)\subset\C^2\setminus\overline{\B}_N$.

Next we shall find a shear automorphism $\phi_2\in\Aut\C^2$ of the form
\begin{equation}
\label{eq:shear}
	\phi_2(z_1,z_2)= \bigl(z_1 + h(z_2),z_2\bigr)
\end{equation}
which is close to the identity on $\C\times (\pi_2(C) \cup \overline D)$
and satisfies  
\[
	\phi_2\bigl( \gamma\cup (\cup_{i=k+1}^\infty\overline\triangle'_i) \bigr)
	\cap\phi_1^{-1}(\overline{\mathbb B}_N)=\emptyset.
\]
The automorphism $\psi=\phi_1\circ\phi_2\in \Aut\C^2$
will then satisfy Lemma \ref{remove}. 

Choose a large number $R>0$ such that 
\[
	\pi_1(\phi_1^{-1}(\overline{\mathbb B}_N)) \subset \D_R \quad
	\hbox{and} \quad
	\pi_2(\phi_1^{-1}(\overline{\mathbb B}_N))\cup \ol D \subset \D_R.
\]
We shall find $\phi_2$ as a composition $\phi_2=\tau_2\circ \tau_1$
of two shears of the same type (\ref{eq:shear}).
The values of the function $h\in\cO(\C)$ in (\ref{eq:shear}) 
on $\C\setminus \D_R$ are unimportant since $\phi_1^{-1}({\ol \B}_N)$ 
projects into $\D_R$. 

Recall that the projection 
$\pi_2\colon \ol X\setminus \pi_2^{-1}(D) \to \C\setminus D$ maps 
$\ol X\setminus \pi_2^{-1}(D)$ bijectively onto a union of 
pairwise disjoint closed wedges with the common vertex at $\infty$
(see Fig.\ \ref{Fig:wedges} below.) 
Hence the geometry of subsets of $\ol X\setminus \pi_2^{-1}(D)$ is the
same as the geometry of their $\pi_2$-projections in $\C\setminus D$,
an observation that will be tacitly used in the sequel.

By \cite[Lemma 1]{Wold} there is an entire function 
$h_1\in \cO(\C)$ which is small on the set $\overline D\cup\pi_2(C)$ and 
takes suitable values on the projected curves $\pi_2(\gamma) \setminus D$ 
so that the shear $\tau_1(z_1,z_2)= (z_1 + h_1(z_2),z_2)$ satisfies  
\[
	\tau_1(\gamma\cup C) \cap \phi_1^{-1}(\overline{\mathbb B}_N)=\emptyset.
\]
Set 
$
	\wt J=\{i\in\N\colon i \geq k+1,\ 
	\pi_2(\overline\triangle'_i)\cap {\overline \D}_R \neq\emptyset\}.
$
Consider the compact set
\[
 	\wt C := \left[\gamma\cap\pi_2^{-1}({\overline \D}_R) \right]
 	\cup \left[\cup_{i\in \wt J} \overline\triangle'_i \right] \subset \ol X.
\]
Let $K''$ be the union of all discs $\overline\triangle_i$
$(i\in \wt J)$ whose images $\overline\triangle'_i=(g\circ f)(\overline\triangle_i)$
satisfy the condition
\[
	\tau_1(\overline\triangle'_i) \cap \phi_1^{-1}(\overline{\mathbb B}_N)
	\ne \emptyset.  
\]
Our choices of $\phi_1$ and $\tau_1$ imply that for every disc
$\overline\triangle_i \subset K''$ the projection 
$\pi_2(\overline\triangle'_i)$ intersects the disc ${\ol\D}_R$ and avoids
the set $\pi_2(C)  \cup \ol D$.
Remark \ref{rem:closed} shows that $K''$ is compact.
Using Lemma \ref{division} we find smoothly bounded discs
$B_1,\ldots, B_l\subset \Omega'$ with pairwise disjoint closures 
whose union $\cup_{j=1}^l B_j$ contains $K''$ and is 
disjoint from $b\Omega' \cup (g\circ f)^{-1}(C)$,
and whose boundaries $bB_j$ belong to $\Omega$.
(Hence every disc $\ol\triangle_i$ for $i>k$ is either completely contained 
in $\cup_{j=1}^l \ol{B}_j$ or else is disjoint from it.) 
It follows that the set
\[
	\wt L:= \cup_{j=1}^l (\pi_2\circ g\circ f)(\ol B_j) \subset\C
\]
is a disjoint union of discs contained in $\C\setminus (\ol D\cup \pi_2(\gamma))$.
Hence  the sets $\wt L$ and $\pi_2(\wt C) \setminus \wt L$ are 
polynomially convex, and so is their union.
(Fig.\ \ref{Fig:wedges} shows $\wt L$ as the union
of black ellipses, while $\pi_2(\wt C) \setminus \wt L$ is shown in gray.)

%
\begin{figure}[ht]
\psset{unit=0.55cm, linewidth=0.5pt}  
\begin{pspicture}(-9,-5)(9,5.5)

%
%
\pscircle[fillstyle=solid,fillcolor=lightgray](0,0){2.5}
\pscircle[linewidth=0.3pt,linestyle=dashed](0,0){5}
\rput(0,0){$D$}
\rput(0,4){$\D_R$}

%
%
\pscurve(2.3,1)(3,1.5)(4,2.5)(6,3)(8.5,4)
\pscurve(2,-1.45)(3,-2)(4,-2.2)(6,-3.5)(8.5,-4.5)
\pscurve(-1.8,1.7)(-3,2.5)(-4,3)(-6,4)(-8.5,4.5)
\pscurve(-2,-1.45)(-3,-2)(-4,-2.2)(-6,-3.5)(-8.5,-4.5)

%
%
\rput(8,0){$\pi_2(\gamma)$}
\psline[linewidth=0.2pt,arrows=->](7.9,0.5)(7,3.3)
\psline[linewidth=0.2pt,arrows=->](7.9,-0.5)(6.8,-3.7)

\rput(-8,0){$\pi_2(\gamma)$}
\psline[linewidth=0.2pt,arrows=->](-7.9,0.5)(-7,4.2)
\psline[linewidth=0.2pt,arrows=->](-7.9,-0.5)(-6.8,-3.8)

%
%
\psellipse*(5,1.5)(0.8,0.45)
\psellipse*(5,-1.2)(0.7,0.35)
\psellipse*(3.7,0)(0.65,0.4)
\psellipse*(5.2,0)(0.65,0.25)

%
%
\pscircle[fillstyle=solid,fillcolor=lightgray,linestyle=dotted](2.3,-0.8){0.5}
\pscircle[fillstyle=solid,fillcolor=lightgray,linestyle=dotted](3.2,-1.3){0.4}
\pscircle[fillstyle=solid,fillcolor=lightgray,linestyle=dotted](3.2,-1.3){0.4}
\pscircle[fillstyle=solid,fillcolor=lightgray,linestyle=dotted](3.5,1.3){0.3}

%
%
\psellipse*(-5,-1)(0.7,0.45)
\psellipse*(-4.7,1)(0.7,0.35)
\psellipse*(-3.7,0)(0.65,0.4)
\psellipse*(-5.2,0)(0.65,0.25)
\psellipse*(-5.1,2.1)(0.8,0.4)

%
%
\pscircle[fillstyle=solid,fillcolor=lightgray,linestyle=dotted](-2.3,0.8){0.5}
\pscircle[fillstyle=solid,fillcolor=lightgray,linestyle=dotted](-3.2,1.3){0.4}
\pscircle[fillstyle=solid,fillcolor=lightgray,linestyle=dotted](-3.2,1.3){0.4}
\pscircle[fillstyle=solid,fillcolor=lightgray,linestyle=dotted](-3.5,-1.3){0.3}
\pscircle[fillstyle=solid,fillcolor=lightgray,linestyle=dotted](-2.3,-0.8){0.6}

\end{pspicture}
\caption{Geometry in the $z_2$-plane.}
\label{Fig:wedges}
\end{figure}
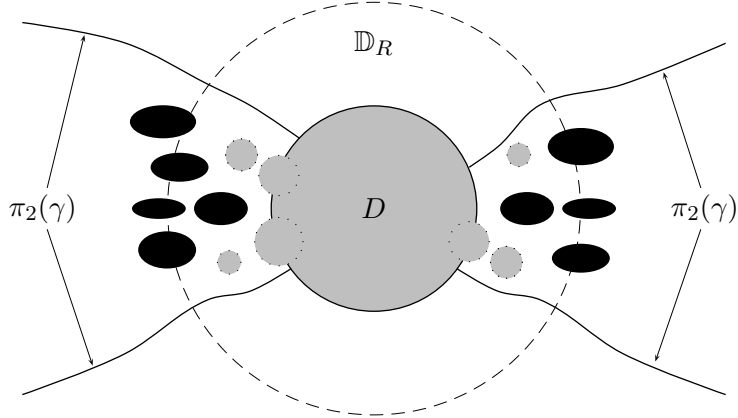

Choose a function $h_2\in \cO(\C)$ such that 
$|h_2|>R$ on $\wt L$ and $|h_2|$ is small on $\pi_2(\wt C) \setminus \wt L$.
Let $\tau_2(z_1,z_2)= (z_1 + h_2(z_2),z_2)$ and 
$\phi_2=\tau_2\circ \tau_1$.
The automorphism $\psi=\phi_1\circ\phi_2\in \Aut\C^2$
then clearly satisfies Lemma \ref{remove}. 

Note that $\phi_2(z_1,z_2)= \bigl(z_1 + h(z_2),z_2)$
with $h=h_1+h_2$, so it is possible to compactify
the construction of $\phi_2$ into one step. 
\end{proof}

%
\section{Proof of Theorem \ref{main}}
\label{sec:proof}
We are now ready to prove Theorem \ref{main}. 
The construction is similar to the proof 
of Majcen's theorem \cite{Majcen-2009} as given in 
\cite[\S 8.10]{F:book}, but the induction scheme is 
altered and improved at several key points. 

Every holomorphic embedding with corners
will be assumed to be of the form (\ref{type}).

Let $\Omega\subset \P^1$ be a domain with countably many 
complementary components, none of which are points.
(We assume that there are infinitely many components, 
for otherwise the result is due to Globevnik and 
Stens\o nes \cite{Globevnik-Stensones}. Our proof
also applies in the latter case, but it could be made much simpler.)  
By the uniformization theorem of He and Schramm \cite{HeSchramm}
we may assume that $\Omega$ is a circled domain. By mapping one of the 
complementary discs in $\P^1\setminus \Omega$ onto the complement 
$\P^1\setminus \D$ of the unit disc $\D$ we may further assume that 
$\Omega=\D\setminus \cup_{j=1}^\infty \ol\triangle_j$, where
$\ol\triangle_j$ are pairwise disjoint closed discs in $\D$.  

We construct a proper holomorphic embedding 
$\Omega\hra \C^2$ by induction. 

Choose an exhaustion 
$
 	\emptyset= K_0 \subset  K_1\subset K_2 \subset 
 	\ldots \subset \cup_{j=1}^\infty K_j=\Omega
$ 
of $\Omega$ by compact, connected, $\cO (\Omega)$-convex sets with 
smooth boundaries, satisfying $K_j \subset \mathring K_{j+1}$ 
for $j=0,1,2,\ldots$.  
These conditions imply that for each index $j\in \N$ 
the set $\wh K_j\setminus K_j \subset  \D$ 
is a union of finitely many open discs, 
i.e., sets homeomorphic to the standard disc.

We begin the induction at $n=0$. Set $\Gamma_0=b\,\D$, $m_0=k_0=0$. 
Pick a point $c_0\in \Gamma_0$ and a number $\epsilon_0>0$. 
At the $n$-th step of the construction we shall obtain 
the following data:
\begin{itemize}
\item integers $m_n, k_n\in \N$,
\item a number $\epsilon_n$ such that $0<\epsilon_n < \frac{1}{2}\,\epsilon_{n-1}$
(the last inequality is void for $n=0$),
\item circles  $\Gamma_j=b\triangle_{i(j)}$ $(j=1,\ldots,k_n)$ 
from the family $\{b\triangle_i\}_{i\in\N}$, at least one in each connected 
component of $\widehat{K_{m_n}}\setminus K_{m_n}$, 
\item the domain $\Omega_n=\mathbb D\bs \cup_{j=1}^{k_{n}} \ol\triangle_{i(j)}$ 
with boundary $b \Omega_n = \cup_{j=0}^{k_n} \Gamma_j$,
\item points $c_j\in \Gamma_j$ for $j=0,\ldots,k_n$, 
\item numbers $\theta_j>0$ $(j=0,\ldots,k_n)$ with  $\sum_{j=0}^{k_n} \theta_j<2\pi$,
\item a holomorphic embedding with corners 
$f_n\colon \ol \Omega_n \hra \C^2$ such that 
the points $c_0,\ldots,c_{k_n}$ are $\pi_1$-exposed
with $\theta_j$-wedges  (see Def.\ \ref{def:exposed-wedge})
and $f_n$ is smooth near $b\Omega_n\setminus\{c_0,\ldots,c_{k_n}\}$,  
\item a rational shear with poles at the exposed points $f_n(c_j)$ 
of $f_n(b\Omega_n)$, 
\[
	g_{n}(z_1,z_2)=\bigg(z_1, z_2+
	\sum_{j=0}^{k_{n}}\frac{\beta_{j}}{z_1-\pi_1(f_{n}(c_j))}\bigg),
\]
such that $(\pi_2\circ g_n\circ f_n)(\Omega_n) \subset\C$
is a union of $\theta_j$-wedges whose closures 
intersect only at their common vertex $\infty\in \P^1$, and 
\item 
an automorphism $\phi_n$ of $\C^2$,
\end{itemize}
such that, setting 
\[
	F_{n-1} = \Phi_{n-1}\circ g_n\circ f_n,
	\qquad	
	\Phi_n=\phi_n\circ \Phi_{n-1} = \phi_n\circ\phi_{n-1} \ldots \circ \phi_1,
\]
the following conditions hold:
\begin{eqnarray}
  && |g_n\circ f_n(x) - g_{n-1}\circ f_{n-1}(x)|<\epsilon_n, \quad 
  x\in K_{m_n},   \label{eqn:pogoj-fn} \\
  && |\Phi_{n-1}\circ g_n\circ f_n(x) - \Phi_{n-1}\circ g_{n-1}\circ f_{n-1}(x)|<\epsilon_n, 
  \quad  x\in K_{m_n},   \label{eqn:pogoj-fn2} \\ 
  && \ol \B_{n-1} \cap F_{n-1}(\Omega_n)\subset F_{n-1}(\mathring{K}_{m_n}),       \label{pogoj_prop} \\
  &&  |\phi_n(z)-z| < \epsilon_n,  \quad z\in \ol{\B}_{n-1} \cup F_{n-1}(K_{m_n}), \label{pogoj_id} \\
	&& 	|\Phi_n \circ g_n \circ f_n(x)|>n, \quad x\in b\Omega_n \cup (\Omega_n\setminus \Omega).\label{pogoj_B}
\end{eqnarray}

\smallskip

\begin{remark}
\label{rem:opening}
Setting $J_{n} = \N\setminus\{i(j)\colon j=1,\ldots,k_{n}\}$, we have that
\[
	\Omega_n = \Omega \,\cup\, \cup_{j\in J_{n}} \ol\triangle_j,\qquad
	\Omega_n\setminus \Omega = \cup_{j\in J_{n}} \ol\triangle_j.
\]
Clearly $\D \supset \Omega_1\supset \Omega_2\supset \cdots \supset \Omega$, 
but the intersection $\bigcap_{j=1}^\infty \Omega_j$ need not equal $\Omega$.
That is, the set of all circles $\Gamma_j$ that get opened up 
in the course of the construction may be a proper subset 
of the family $\{b\triangle_i\}_{i\in \Z_+}$ 
of all boundary circles of $\Omega$. The only reason for opening 
a boundary circle contained in $\widehat{K_{m_n}}\setminus K_{m_n}$ 
is to insure that the image of $K_{m_n}$ in $\C^2$ becomes
polynomially convex; see (\ref{eq:polyconv}) below.
\qed\end{remark}

We begin the induction at $n=0$ by choosing an 
embedding $f_0(\zeta)=(\tau_0(\zeta),0)$ of $\ol\D$ in $\C\times\{0\}\subset\C^2$
with a $\theta_0$-wedge at the point $c_0\in\Gamma_0=b\D$
(see Theorem \ref{exposing}). We also choose a shear 	
\[
	g_{0}(z_1,z_2)=\left(z_1, z_2+ \frac{\beta_{0}}{z_1-\pi_1\circ f_{0}(c_0)}\right),
\]
sending the exposed point $\pi_1\circ f_0(c_0)=\tau_0(c_0)$ to infinity.
Let $\phi_0=\Phi_0=\Phi_{-1}= \Id$. 
Conditions (\ref{eqn:pogoj-fn}) -- (\ref{pogoj_id}) 
are then vacuous for $n=0$ (recall that $K_0=\emptyset$), 
and $(\ref{pogoj_B})$ is satisfied after a small translation
of the embedding $g_0\circ f_0\colon \ol\D\setminus \{c_0\} \hra\C^2$ 
which removes the image off the origin.

We now explain the inductive step $n\rightarrow n+1$.
By (\ref{pogoj_B}) there exists an integer $m_{n+1}>m_n$ such that
\begin{equation}
\label{eq:goodintersection}
	\ol \B_n \cap \bigl(\Phi_n\circ g_n\circ f_n(\Omega_n)\bigr) \subset
	\Phi_n\circ g_n\circ f_n(\mathring{K}_{m_{n+1}}).
\end{equation}
By the inductive hypothesis the polynomial hull $\wh K_{m_{n+1}}$ 
contains the boundary circles $\Gamma_j\subset b\Omega$ for $1\le j \le k_{n}$.
(This is vacuous if $n=0$.)
In each of the (finitely many) connected components
of $\widehat K_{m_{n+1}}\setminus K_{m_{n+1}}$ which does not contain
any of the above circles we pick another boundary circle of $\Omega$
(such exists since the set $K_{m_{n+1}}$ is $\cO(\Omega)$-convex); 
we label these additional curves $\Gamma_{k_n +1},\ldots, \Gamma_{k_{n+1}}$. 
As before, we have $\Gamma_j=b\triangle_{i(j)}$ for some index $i(j)$. Let 
\[
	\Omega_{n+1}=\D \bs \cup_{j=1}^{k_{n+1}} \ol\triangle_{i(j)}.
\]
Setting $J_{n+1} = \N\setminus\{i(j)\colon j=1,\ldots,k_{n+1}\}$, we have that
\[
	\Omega_{n+1} = \Omega \cup \cup_{j\in J_{n+1}} \ol\triangle_j.
\]
Each of these additional curves will now be opened up. 
Pick a point $c_j\in \Gamma_j$ for each 
$j=k_n+1,\ldots, k_{n+1}$ and positive numbers 
$\theta_{k_n+1},\ldots,\theta_{k_{n+1}}$
such that  $\sum_{j=0}^{k_{n+1}} \theta_j<2\pi$.
Also choose a number $\epsilon_{n+1}\in(0,\epsilon_n/2)$
such that any holomorphic map
$h\colon \Omega\to \C^2$ satisfying  
$||h-g_n\circ f_n||_{K_{m_{n+1}}} < 2\epsilon_{n+1}$
is an embedding on $K_{m_n}$.
Theorem \ref{exposing} furnishes a holomorphic embedding 
$f_{n+1}\colon \ol \Omega_{n+1}\hra \C^2$ with corners
such that $f_{n+1}$ agrees with $f_n$ 
to the second order at each of the points $c_0,\ldots, c_{k_{n}}$,
it additionally makes the boundary points $c_{k_n+1}, \ldots, c_{k_{n+1}}$
$\pi_1$-exposed with $\theta_j$-wedges,
and it approximates $f_n$ as close as desired
outside of small neighborhoods of these points.
The image $f_{n+1}(\ol \Omega_{n+1})$ stays as close as desired 
to the union of $f_n(\ol \Omega_{n+1})$ with the family 
of arcs that were attached to this set in order 
to expose the points $c_{k_n+1}, \ldots, c_{k_{n+1}}$.
In particular, we insure that none of the complex lines 
$z_1=\pi_1\circ f_{n+1}(c_j)$ for $j=k_{n}+1,\ldots,k_{n+1}$ 
intersects the set $\Phi_n^{-1}(\ol{\B}_n)$. 
The rational shear 
\[
	g_{n+1}(z_1,z_2)= g_n(z_1,z_2) +  
	\bigg(0,\sum_{j=k_{n}+1}^{k_{n+1}} 
	\frac{\beta_{j}}{z_1-\pi_1(f_{n+1}(c_j))} \bigg)
\]
sends the exposed points  
$f_{n+1}(c_0), \ldots, f_{n+1}(c_{k_{n+1}})$ to infinity. 
A suitable choice of the arguments of the numbers $\beta_j\in\C^*$ 
for $j=k_{n}+1,\ldots,k_{n+1}$ insures that, in a neigborhood of infinity,  
$(\pi_2\circ g_{n+1}\circ f_{n+1})(\overline\Omega_{n+1})$ 
is a union of pairwise disjoint $\theta_j$-wedges with 
the common vertex at $\infty\in\P^1$; at the same time 
the absolute values $|\beta_j|>0$ can be chosen arbitrarily small
in order to obtain good approximation of $g_n$ by $g_{n+1}$. 

Set $F_{n}=\Phi_n\circ g_{n+1} \circ f_{n+1}$.
If the approximations of $f_n$, $g_n$ by $f_{n+1}$,
$g_{n+1}$, respectively, were close enough, then 
the conditions (\ref{eqn:pogoj-fn})--(\ref{pogoj_prop}) 
hold with $n$ replaced by $n+1$. 

Since every connected component of $\wh{K}_{m_{n+1}} \setminus K_{m_{n+1}}$
contains at least one of the points $c_1,\ldots,c_{m_{n+1}}$ which
$F_n$ sends to infinity, 
the set $F_n(K_{m_{n+1}}) \subset \C^2$ is polynomially convex.
(See \cite[Prop.\ 3.1]{Wold} for the details of this argument.)
From (\ref{eq:goodintersection}) we also infer that 
$
	\ol \B_n \cap F_n(\Omega_{n+1}) \subset F_n(\mathring{K}_{m_{n+1}})
$
provided that the approximations were close enough.
It follows that the set
\begin{equation}
\label{eq:polyconv}
	L_n : =\ol \B_n \cup F_n(K_{m_{n+1}}) \subset \C^2
\end{equation}
is polynomially convex. 

Now comes the last, and the main step in the induction: 
We use Lemma \ref{remove} to find an automorphism $\phi_{n+1}\in \Aut\C^2$
which satisfies conditions (\ref{pogoj_id}) and (\ref{pogoj_B}) 
with $n$ replaced by $n+1$.  We look for $\phi_{n+1}$ of the form
\[
	\phi_{n+1}=\Phi_n\circ\psi\circ\Phi^{-1}_n,\quad \psi\in\Aut\C^2.
\]
(Therefore $\Phi_{n+1}=\phi_{n+1}\circ \Phi_n= \Phi_n\circ\psi$.)
Pick a small constant $\delta>0$ such that for any pair of points
$z,z'\in\C^2$, with $z\in \Phi_n^{-1}(L_n)$ 
and $|z-z'|<\delta$, we have 
$|\Phi_n(z)-\Phi_n(z')| <\epsilon_{n+1}$. 
(Such $\delta$ exists by continuity of $\Phi_{n}$.) 
We also pick a large constant $R>0$
such that $|\Phi_n(z)|>n+1$ for all $z\in \C^2$ with $|z|>R$.
(Equivalently, $\Phi_n^{-1}(\ol{\B}_n) \subset\ol{\B}_R$.) 
Since the set $\Phi_n^{-1}(L_n)$ is polynomially convex,
Lemma \ref{remove} furnishes an automorphism $\psi\in\Aut\C^2$ 
satisfying the following two conditions:
\begin{itemize}
\item[\rm (4.4')] $|\psi(z)-z|<\delta$ for 
$z\in \Phi_n^{-1}\bigl(L_n)$, and
\item[\rm (4.5')] $|\psi(z)|>R$ for 
$z\in g_{n+1}\circ f_{n+1}\bigl(b\Omega_{n+1} \,\cup\, 
\cup_{j\in J_{n+1}} \ol\triangle_j\bigr)$.
\end{itemize}
By (\ref{pogoj_prop}) (applied with $n+1$) the two sets appearing in these conditions
are disjoint. It is now immediate that $\phi_{n+1}$ 
satisfies conditions (\ref{pogoj_id}), (\ref{pogoj_B}).

This completes the induction step, so the induction may proceed.

We now conclude the proof. 
By (\ref{eqn:pogoj-fn}) and the choice of the numbers 
$\epsilon_n>0$ we see that the limit map
$
	G=\lim_{n\to \infty} g_n\circ f_n\colon \Omega \to  \C^2
$
is a holomorphic embedding. 
Condition (\ref{pogoj_id}) implies that the sequence 
$\Phi_n\in\Aut\C^2$ converges on the domain  
$O = \cup_{n=2}^\infty \Phi_n^{-1} (\ol\B_{n-1}) \subset \C^2$
to a Fatou-Bieberbach map $\Phi=\lim_{n\to\infty}\Phi_n \colon O \to\C^2$,
i.e., a biholomorphic map of $O$ onto $\C^2$
(c.f.\ \cite[Corollary 4.4.2]{F:book}).  
Conditions (\ref{eqn:pogoj-fn2}) and (\ref{pogoj_id}) show that the sequence 
$\Phi_n$ converges on $G(\Omega)$, so $G(\Omega) \subset O$.
From (\ref{pogoj_prop}) and (\ref{pogoj_B}) we see that 
$G$ embeds $\Omega$ properly into $O$. 
Hence the map 
\[
	F=\Phi\circ G = \lim_{n\to\infty} \Phi_n \circ g_n\circ f_n \colon \Omega \hra \C^2
\]
is a proper holomorphic embedding of $\Omega$ into $\C^2$.
\qed

\begin{remark}
If we choose an initial holomorphic embedding 
$f_0\colon \ol\D\hra \C^2$, a compact set $K=K_0\subset \Omega$
and a number $\epsilon>0$, then the above construction 
is easily modified to yield a proper holomorphic embedding $F \colon \Omega\hra\C^2$ 
satisfying $||F-f||_K <\epsilon$. Furthermore, we can choose 
$F$ to agree with $f$ at finitely many points of $\Omega$.
All these additions are standard.
\end{remark}

%
\section{Domains with punctures}
\label{sec:punctures}
Theorem \ref{main} can be extended 
to domains $\Omega$ in $\P^1$ with certain boundary punctures. 
By a {\em puncture} we mean a connected component of  $\P^1\setminus \Omega$ which is a point.
We say that a domain $\Omega\subset\mathbb P^1$ is a \emph{generalized circled domain} if 
each complementary component is either a round disc or a puncture.  
By He and Schramm \cite{HeSchramm}, any domain in $\P^1$ 
with at most countably many boundary components is conformally equivalent 
to a generalized circled domain.  

Our main result in this direction is the following.

%
%
%
%
\begin{theorem}
\label{main2}
Let $\Omega$ be a generalized circled domain in $\mathbb P^1$.
If all but finitely many punctures in 
the complement $K:=\mathbb P^1\setminus\Omega$ 
are limit points of discs in $K$, then $\Omega$ embeds 
properly holomorphically in $\mathbb C^2$.
\end{theorem}

\begin{corollary}
\label{cor:circled-C}
If $\Omega$ is a circled domain in $\C$ or in $\P^1$ 
and $p_1,\ldots,p_l\in \Omega$ is an arbitrary 
finite set of points in $\Omega$, then the domain 
$\Omega\setminus \{p_1,\ldots,p_l\}$ admits a proper 
holomorphic embedding in $\C^2$.
\end{corollary}

By He and Schramm, Corollary \ref{cor:circled-C} 
also holds for $\Omega\setminus \{p_1,\ldots,p_l\}$,
where $\Omega \subset \P^1$ is a domain 
as in Theorem \ref{main}.

\begin{proof}[Proof of Theorem \ref{main2}.] 
We make the following modifications in the proof of Theorem \ref{main}. 
We may assume as before that $\Omega$ 
%
%
is contained in the unit 
disc $\D$, with $\Gamma_0=b\D$ being one of its boundary components.
Let $f_0\colon \Omega\hra\C^2$ be the embedding $\zeta\mapsto(\zeta,0)$. 
Assume that $p_1,\ldots,p_l \in b\Omega$ are the finitely many punctures 
which do not belong to the cluster set of $\cup_{i} \ol \Delta_i$.
A rational shear 
$g_0(z_1,z_2)=\left( z_1,z_2 + \sum_{j=1}^l \frac{\beta_j}{z_1-p_j} \right)$
sends the points $p_1,\ldots,p_l$ to infinity. We then 
apply the rest of the proof exactly as before, 
insuring at each step of the inductive construction
that the embedding with corners $f_n\colon \ol\Omega_n \hra\C^2$
agrees with $f_0$ at the points $p_1,\ldots,p_l$ and leaves
these points $\pi_1$-exposed, and the shear 
$g_n$ has poles at these points. The coordinate 
projection $\pi_2 \colon \ol X_n= g_n\circ f_n(\ol\Omega_n) \to\C$ 
is no longer injective near infinity due to the poles of $g_n$
at the points $p_1,\ldots,p_l$. However, since the discs $\ol \Delta_i$ do 
not accumulate on any of the points $p_1,\ldots,p_l$, the discs 
$(g_n\circ f_n)(\ol\Delta_i) \subset X_n$ which approach infinity
are still mapped bijectively to a finite union of pairwise disjoint wedges
at $\infty$, and the additional sheets of the projection 
$\pi_2 \colon \ol X_n \to\C$ 
are irrelevant for the construction of the automorphism
which removes the discs and the boundary curves of $X_n$ 
out of a given ball in $\C^2$.  

The remaining punctures $p_\lambda$ in $b\Omega$ 
(a possibly uncountable set) 
can be treated in the same way as the complementary discs.
Indeed, since each of these points is the limit point of the 
sequence of discs $\Delta_i$, every connected component of the set 
$\widehat K_{m}\setminus K_m$ (where $K_m$ is a sequence of
compacts exhausting the domain $\Omega$, see \S \ref{sec:proof})
which contains one of these punctures $p_\lambda$ also contains 
a disc $\ol\Delta_i$. By exposing a boundary point 
of $\Delta_i$ and removing it to infinity by a rational
shear we thus insure that the image of $p_\lambda$ does not belong 
to the polynomial hull of the image of $K_m$ in $\C^2$.
(See Remark \ref{rem:opening}.) 
The conclusion of Remark \ref{rem:closed} is still valid,
and hence the arguments in the proof of Theorem \ref{main} 
concerning moving compact sets by automorphisms of $\C^2$ 
still apply without any changes.  
\end{proof}

%
%
%
%

\begin{example}
\label{ex:referee}
Assume that $E\subset\mathbb P^1$ is any compact totally disconnected set.
(In particular, $E$ could be a Cantor set). Then we may choose a sequence 
of pairwise disjoint closed round discs $\overline \triangle_j \subset \P^1\setminus E$ 
such that each point of $E$ is a cluster point of the sequence $\{\triangle_j\}$
and such that $\Omega:=\mathbb P^1\setminus (E\cup(\cup_j\overline\triangle_j))$
is a domain.  Then $\Omega$ embeds properly in $\mathbb C^2$.

There exists a Cantor set in $\P^1$ whose complement
embeds properly holomorphically into $\C^2$ (Orevkov \cite{Orevkov}),
but it is an open problem whether this holds for each Cantor set.
\qed\end{example}

{\bf Acknowledgements.} 
We express our sincere thanks to an anonymous referee who 
contributed many useful remarks and suggestions which helped us to improve 
the presentation. We also thank Frank Kutzschebauch for 
communicating to us his observation that our proof of Theorem \ref{main}
also yields Theorem \ref{main2}. 
Research of the first named author was supported by grants 
P1-0291 and J1-2152, Republic of Slovenia.
Research of the second named author was supported 
by the grant NFR-209751/F20 from the Norwegian Research Council.

\bibliographystyle{amsplain}

\end{document}